\documentclass[a4paper,11pt]{amsart}
\usepackage{amsmath,amssymb,amsfonts,latexsym}

\newtheorem{theorem}{Theorem}[section]
\newtheorem{lemma}[theorem]{Lemma}
\newtheorem{corollary}[theorem]{Corollary}

\input{tcilatex}

\begin{document}
\title[\textbf{On a discontinuous fourth-order boundary value problem }]{%
\textbf{Asymptotic expressions of eigenvalues and fundamental solutions of a
discontinuous fourth-order boundary value problem}}
\author[\textbf{E. \c{S}en}]{\textbf{Erdo\u{g}an \c{S}en}}
\address{\textbf{Department of Mathematics, Faculty of Science and Letters,
Nam\i k Kemal University, 59030 Tekirda\u{g}, TURKEY}}
\email{\textbf{erdogan.math@gmail.com}}
\author[\textbf{S. Araci}]{\textbf{Serkan Araci}}
\address{\textbf{University of Gaziantep, Faculty of Science and Arts,
Department of Mathematics, 27310 Gaziantep, TURKEY}}
\email{\textbf{mtsrkn@hotmail.com}}
\author[\textbf{M. Acikgoz}]{\textbf{Mehmet Acikgoz}}
\address{\textbf{University of Gaziantep, Faculty of Science and Arts,
Department of Mathematics, 27310 Gaziantep, TURKEY}}
\email{\textbf{acikgoz@gantep.edu.tr}}

\begin{abstract}
In the present paper, we deal with a fourth-order boundary value problem
problem with eigenparameter dependent boundary conditions and transmission
conditions at a interior point. A self-adjoint linear operator $A$ is
defined in a suitable Hilbert space $H$ such that the eigenvalues of such a
problem coincide with those of $A$. Following Mukhtarov and his students
methods [2,4,6] we obtain asymptotic formulae for its eigenvalues and
fundamental solutions. Our applications possess a number of interesting
properties for studying in boundary value problems which we state in this
paper.

\vspace{2mm}\noindent \textsc{2010 Mathematics Subject Classification.}
34L20, 35R10.

\vspace{2mm}

\end{abstract}

\maketitle




\section{\textbf{Introduction}}


It is well-known that many topics in mathematical physics require the
investigation of eigenvalues and eigenfunctions of Sturm-Liouville type
boundary value problems. In recent years, more and more researchers are
interested in the discontinuous Sturm-Liouville problem (see [1-6]). Various
physics applications of this kind problem are found in many literatures,
including some boundary value problem with transmission conditions that
arise in the theory of heat and mass transfer (see [8,9]). The literature on
such results is voluminous and we refer to [1-11].

Fourth-order discontinuous boundary value problems with eigen-dependent
boundary conditions and with two supplementary transmission conditions at
the point of discontinuity have been investigated in [12,13]. Note that
discontinuous Sturm-Liouville problems with eigen-dependent boundary
conditions and with four supplementary transmission conditions at the points
of discontinuity have been investigated in [3].

In this study, we shall consider a fourth-order differential equation 
\begin{equation}
Lu:=\left( a\left( x\right) u^{\prime\prime}\left( x\right) \right)
^{\prime\prime}+q(x)u(x)=\lambda u(x)  \tag*{(1.1)}
\end{equation}
on $I=\left[ -1,0\right) \cup\left( 0,1\right] ,$ with boundary conditions
at $x=-1$%
\begin{align}
L_{1}u & :=u\left( -1\right) =0,  \tag*{(1.2)} \\
L_{2}u & :=\beta_{1}u^{\prime}(-1)+\beta_{2}u^{\prime\prime}\left( -1\right)
=0,  \tag{1.3}
\end{align}
with the six transmission conditions at the points of discontinuity $x=0,$

\begin{align}
L_{3}u & :=u\left( 0+\right) -u\left( 0-\right) =0,  \tag*{(1.4)} \\
L_{4}u & :=u^{\prime}\left( 0+\right) -u^{\prime}\left( 0-\right) =0, 
\tag{1.5} \\
L_{5}u & :=u^{\prime\prime}\left( 0+\right) -u^{\prime\prime}\left(
0-\right) +\lambda\delta_{1}u^{\prime}\left( 0-\right) =0,  \tag{1.6} \\
L_{6}u & :=u^{\prime\prime\prime}\left( 0+\right) -u^{\prime\prime\prime
}\left( 0-\right) +\lambda\delta_{2}u\left( 0-\right) =0,  \tag{1.7}
\end{align}
and the eigen-dependent boundary conditions at $x=1$%
\begin{align}
L_{7}u & :=\lambda u\left( 1\right) +u^{\prime\prime\prime}(1)=0, 
\tag*{(1.8)} \\
L_{8}u & :=\lambda u^{\prime}\left( 1\right) +u^{\prime\prime }(1)=0, 
\tag{1.9}
\end{align}
where $a\left( x\right) =a_{1}^{4},$ for $x\in\left[ -1,0\right) ,$ $a\left(
x\right) =a_{2}^{4},$ for $x\in\left( 0,1\right] ,$ $a_{1}>0$ and $a_{2}>0$
are given real numbers, $q(x)$ is a given real-valued function continuous in 
$\left[ -1,0\right) \cup\left( 0,1\right] $ and has a finite limit $%
q(0\pm)=\lim_{x\rightarrow0}\pm q(x)$; $\lambda$ is a complex eigenvalue
parameter; $\beta_{i},\delta_{i}$ $\left( i=1,2\right) $ are real numbers
and $\left\vert \beta_{1}\right\vert +\left\vert \beta_{2}\right\vert \neq0,$
$\left\vert \delta_{1}\right\vert +\left\vert \delta_{2}\right\vert \neq0.$

\section{\textbf{Preliminaries}}

Firstly we define the inner product in $L^{2}$ for every $f,g\in L^{2}\left(
I\right) $ as%
\begin{equation*}
\left\langle f,g\right\rangle _{1}=\frac{1}{a_{1}^{4}}\int_{-1}^{0}f_{1}%
\overline{g_{1}}dx+\frac{1}{a_{2}^{4}}\int_{0}^{1}f_{2}\overline{g_{2}}dx,
\end{equation*}%
where $f_{1}(x)=f(x)\left\vert _{\left[ -1,0\right) }\right. ,$ $%
f_{2}(x)=f(x)\left\vert _{\left( 0,1\right] }\right. $. It is easy to see
that $\left( L^{2}\left( I\right) ,\left[ \cdot ,\cdot \right] \right) $ is
a Hilbert space. Now we define the inner product in the direct sum of spaces 
$L^{2}\left( I\right) \oplus 
\mathbb{C}
\oplus 
\mathbb{C}
\oplus 
\mathbb{C}
_{\delta _{1}}\oplus 
\mathbb{C}
_{\delta _{2}}$ by%
\begin{equation*}
\left[ F,G\right] :=\left\langle f,g\right\rangle _{1}+\left\langle
h_{1},k_{1}\right\rangle +\left\langle h_{2},k_{2}\right\rangle
+\left\langle h_{3},k_{3}\right\rangle +\left\langle h_{4},k_{4}\right\rangle
\end{equation*}%
for 
\begin{equation*}
F:=\left( f,h_{1},h_{2},h_{3},h_{4}\right) ,G:=\left(
g,k_{1},k_{2},k_{3},k_{4}\right) \in L^{2}\left( I\right) \oplus 
\mathbb{C}
\oplus 
\mathbb{C}
\oplus 
\mathbb{C}
_{\delta _{1}}\oplus 
\mathbb{C}
_{\delta _{2}}\text{.}
\end{equation*}
Then $Z:=\left( L^{2}\left( I\right) \oplus 
\mathbb{C}
\oplus 
\mathbb{C}
\oplus 
\mathbb{C}
_{\delta _{1}}\oplus 
\mathbb{C}
_{\delta _{2}},\left[ \cdot ,\cdot \right] \right) $ is the direct sum of
modified Krein spaces. A fundamental symmetry on the Krein space is given by%
\begin{equation*}
J:=\left[ 
\begin{array}{ccccc}
J_{0} & 0 & 0 & 0 & 0 \\ 
0 & 1 & 0 & 0 & 0 \\ 
0 & 0 & 1 & 0 & 0 \\ 
0 & 0 & 0 & sgn\delta _{1} & 0 \\ 
0 & 0 & 0 & 0 & sgn\delta _{2}%
\end{array}%
\right] ,
\end{equation*}%
where $J_{0}:L^{2}\left( I\right) \rightarrow L^{2}\left( I\right) $ is
defined by $\left( J_{0}f\right) \left( x\right) =f\left( x\right) .$ We
define a linear operator $A$ in $Z$ by the domain of definition%
\begin{gather*}
D\left( A\right) :=\left( f\text{, }h_{1}\text{, }h_{2}\text{, }h_{3}\text{, 
}h_{4}\right) \in Z\mid f_{1}^{(i)}\in AC_{loc}\left( \left( -1,0\right)
\right) \text{, }f_{2}^{(i)}\in AC_{loc}\left( \left( 0,1\right) \right) 
\text{, }i=\overline{0,3}, \\
Lf\in L^{2}\left( I\right) ,\text{ }L_{k}f=0,\text{ }k=\overline{1,6},\text{ 
}h_{1}=f\left( 1\right) ,\text{ }h_{2}=f^{\prime }(1),\text{ }h_{3}=-\delta
_{1}f^{\prime }\left( 0\right) ,\text{ }h_{4}=-\delta _{2}f(0), \\
AF=\left( Lf,\text{ }-f^{\prime \prime \prime }(1),\text{ }-f^{\prime \prime
}\left( 1\right) ,\text{ }f^{\prime \prime }\left( 0+\right) -f^{\prime
\prime }\left( 0-\right) ,\text{ }f^{\prime \prime \prime }\left( 0+\right)
-f^{\prime \prime \prime }\left( 0-\right) \right) , \\
F=\left( f,\text{ }f(1),\text{ }f^{\prime }(1),-\delta _{1}f^{\prime }(0),%
\text{ }-\delta _{2}f(0)\right) \in D\left( A\right) .
\end{gather*}%
Consequently, the considered problem (1.1)-(1.9) can be rewritten in
operator form as

\begin{equation*}
AF=\lambda F\text{,}
\end{equation*}%
i.e., the problem (1.1)-(1.9) can be considered as the eigenvalue problem
for the operator $A$. Then, we can write the following conclusions:

\begin{theorem}
\textit{The eigenvalues and eigenfunctions of the problem (1.1)-(1.9) are
defined as the eigenvalues and the first components of the corresponding
eigenelements of the operator }$A$\textit{\ respectively.}
\end{theorem}

\begin{theorem}
\textit{The operator }$A$ \textit{is self-adjoint in Krein space} $Z$ (cf.
Theorem 2.2 of $\left[ 10\right] $).
\end{theorem}

\section{\textbf{Fundamental Solutions}}

\begin{lemma}
\textit{Let the real-valued function }$q\left( x\right) $\textit{\ be
continuous in }$\left[ -1,1\right] $ and $f_{i}\left( \lambda \right) $%
\textit{\ }$\left( i=1,4\right) $ \textit{are given entire functions. Then
for any }$\lambda \in 
\mathbb{C}
$\textit{\ the equation }%
\begin{equation*}
\left( a\left( x\right) u^{\prime \prime }\left( x\right) \right) ^{\prime
\prime }+q(x)u(x)=\lambda u(x),\text{ \ }x\in I
\end{equation*}%
\textit{has a unique solution }$u=u\left( x,\lambda \right) $\textit{\ such
that }%
\begin{equation*}
u\left( -1\right) =f_{1}\left( \lambda \right) \text{, }u^{^{\prime }}\left(
-1\right) =f_{2}\left( \lambda \right) \text{, }u^{\prime \prime }\left(
-1\right) =f_{3}\left( \lambda \right) \text{, }u^{\prime \prime \prime
}\left( -1\right) =f_{4}\left( \lambda \right)
\end{equation*}%
\begin{equation*}
\left( \text{or }u\left( 1\right) =f_{1}\left( \lambda \right) ,\text{ }%
u^{^{\prime }}\left( 1\right) =f_{2}\left( \lambda \right) ,\text{ }%
u^{\prime \prime }\left( 1\right) =f_{3}\left( \lambda \right) ,\text{ }%
u^{\prime \prime \prime }\left( 1\right) =f_{4}\left( \lambda \right)
\right) .
\end{equation*}%
\textit{and for each }$x\in \left[ -1,1\right] ,$\textit{\ }$u\left(
x,\lambda \right) $\textit{\ is an entire function of }$\lambda .$
\end{lemma}

\begin{proof}
Let\textbf{\ }$\phi _{11}\left( x,\lambda \right) $ be the solution of Eq.
(1.1) on $\left[ -1,0\right) $ which satisfies the initial conditions%
\begin{gather*}
\phi _{11}\left( -1\right) =0,\text{ }\phi _{11}^{\prime }\left( -1\right)
=\phi _{11}^{\prime \prime }\left( -1\right) =0,\text{ } \\
\phi _{11}^{\prime \prime \prime }\left( -1\right) =-1.
\end{gather*}%
By virtue of Lemma 3.1, after defining this solution, we may define the
solution $\phi _{12}\left( x,\lambda \right) $ of Eq. (1.1) on $\left( 0,1%
\right] $ by means of the solution $\phi _{11}\left( x,\lambda \right) $ by
the initial conditions%
\begin{gather}
\phi _{12}\left( 0\right) =\phi _{11}\left( 0\right) ,\text{ }\phi
_{12}^{\prime }\left( 0\right) =\phi _{11}^{\prime }\left( 0\right) ,\text{ }%
\phi _{12}^{\prime \prime }\left( 0\right) =\phi _{11}^{\prime \prime
}\left( 0\right) -\lambda \delta _{1}\phi _{11}^{\prime }\left( 0\right) , 
\notag \\
\phi _{12}^{\prime \prime \prime }\left( 0\right) =\phi _{11}^{\prime \prime
\prime }\left( 0\right) -\lambda \delta _{2}\phi _{11}\left( 0\right) . 
\tag{3.1}
\end{gather}%
After defining this solution, we may define the solution $\phi _{21}\left(
x,\lambda \right) $ of Eq. (1.1) on $\left[ -1,0\right) $ which satisfies
the initial conditions%
\begin{equation}
\phi _{21}\left( -1\right) =0,\text{ }\phi _{21}^{\prime }\left( -1\right)
=\beta _{2},\text{ }\phi _{21}^{\prime \prime }\left( -1\right) =-\beta _{1},%
\text{ }\phi _{21}^{\prime \prime \prime }\left( -1\right) =0.  \tag{3.2}
\end{equation}%
After defining this solution, we may define the solution $\phi _{22}\left(
x,\lambda \right) $ of Eq. (1.1) on $\left( 0,1\right] $ by means of the
solution $\phi _{21}\left( x,\lambda \right) $ by the initial conditions%
\begin{gather}
\phi _{22}\left( 0\right) =\phi _{21}\left( 0\right) ,\text{ }\phi
_{22}^{\prime }\left( 0\right) =\phi _{21}^{\prime }\left( 0\right) ,\text{ }%
\phi _{22}^{\prime \prime }\left( 0\right) =\phi _{21}^{\prime \prime
}\left( 0\right) -\lambda \delta _{1}\phi _{21}^{\prime }\left( 0\right) , 
\notag \\
\phi _{22}^{\prime \prime \prime }\left( 0\right) =\phi _{21}^{\prime \prime
\prime }\left( 0\right) -\lambda \delta _{2}\phi _{21}\left( 0\right) . 
\tag{3.3}
\end{gather}%
Analogically we shall define the solutions $\chi _{11}\left( x,\lambda
\right) $ and $\chi _{12}\left( x,\lambda \right) $ by the initial conditions%
\begin{gather}
\chi _{12}\left( 1\right) =-1,\text{ }\chi _{12}^{\prime }\left( 1\right)
=\chi _{12}^{\prime \prime }\left( 1\right) =0,\chi _{12}^{\prime \prime
\prime }\left( 1\right) =\lambda ,\text{ }\chi _{11}\left( 0\right) =\chi
_{12}\left( 0\right) ,  \notag \\
\chi _{11}^{\prime }\left( 0\right) =\chi _{12}^{\prime }\left( 0\right) ,%
\text{ }\chi _{11}^{\prime \prime }\left( 0\right) =\chi _{12}^{\prime
\prime }\left( 0\right) +\lambda \delta _{1}\chi _{12}^{\prime }\left(
0\right) ,\text{ }  \tag{3.4} \\
\chi _{11}^{\prime \prime \prime }\left( 0\right) =\chi _{12}^{\prime \prime
\prime }\left( 0\right) +\lambda \delta _{2}\chi _{12}\left( 0\right) . 
\notag
\end{gather}

Moreover, we shall define the solutions $\chi _{21}\left( x,\lambda \right) $
and $\chi _{22}\left( x,\lambda \right) $ by the initial conditions%
\begin{gather}
\chi _{22}\left( 1\right) =0,\text{ }\chi _{22}^{\prime }\left( 1\right) =-1,%
\text{ }\chi _{22}^{\prime \prime }\left( 1\right) =\lambda ,\text{ }\chi
_{22}^{\prime \prime \prime }\left( 1\right) =0,\text{ }\chi _{21}\left(
0\right) =\chi _{22}\left( 0\right) ,  \notag \\
\chi _{21}^{\prime }\left( 0\right) =\chi _{22}^{\prime }\left( 0\right) ,%
\text{ }\chi _{21}^{\prime \prime }\left( 0\right) =\chi _{22}^{\prime
\prime }\left( 0\right) +\lambda \delta _{1}\chi _{22}^{\prime }\left(
0\right) ,\text{ }  \tag{3.5} \\
\chi _{21}^{\prime \prime \prime }\left( 0\right) =\chi _{22}^{\prime \prime
\prime }\left( 0\right) +\lambda \delta _{2}\chi _{22}\left( 0\right) . 
\notag
\end{gather}%
Let us consider the Wronskians%
\begin{equation*}
W_{1}\left( \lambda \right) :=\left\vert 
\begin{array}{cccc}
\phi _{11}\left( x,\lambda \right) & \phi _{21}\left( x,\lambda \right) & 
\chi _{11}\left( x,\lambda \right) & \chi _{21}\left( x,\lambda \right) \\ 
\phi _{11}^{\prime }\left( x,\lambda \right) & \phi _{21}^{\prime }\left(
x,\lambda \right) & \chi _{11}^{\prime }\left( x,\lambda \right) & \chi
_{21}^{\prime }\left( x,\lambda \right) \\ 
\phi _{11}^{\prime \prime }\left( x,\lambda \right) & \phi _{21}^{\prime
\prime }\left( x,\lambda \right) & \chi _{11}^{\prime \prime }\left(
x,\lambda \right) & \chi _{21}^{\prime \prime }\left( x,\lambda \right) \\ 
\phi _{11}^{\prime \prime \prime }\left( x,\lambda \right) & \phi
_{21}^{\prime \prime \prime }\left( x,\lambda \right) & \chi _{11}^{\prime
\prime \prime }\left( x,\lambda \right) & \chi _{21}^{\prime \prime \prime
}\left( x,\lambda \right)%
\end{array}%
\right\vert
\end{equation*}

and%
\begin{equation*}
W_{2}\left( \lambda\right) :=\left\vert 
\begin{array}{cccc}
\phi_{12}\left( x,\lambda\right) & \phi_{22}\left( x,\lambda\right) & 
\chi_{12}\left( x,\lambda\right) & \chi_{22}\left( x,\lambda\right) \\ 
\phi_{12}^{\prime}\left( x,\lambda\right) & \phi_{22}^{\prime}\left(
x,\lambda\right) & \chi_{12}^{\prime}\left( x,\lambda\right) & \chi
_{22}^{\prime}\left( x,\lambda\right) \\ 
\phi_{12}^{\prime\prime}\left( x,\lambda\right) & \phi_{22}^{\prime\prime
}\left( x,\lambda\right) & \chi_{12}^{\prime\prime}\left( x,\lambda\right) & 
\chi_{22}^{\prime\prime}\left( x,\lambda\right) \\ 
\phi_{12}^{\prime\prime\prime}\left( x,\lambda\right) & \phi_{22}^{\prime%
\prime\prime}\left( x,\lambda\right) & \chi_{12}^{\prime\prime \prime}\left(
x,\lambda\right) & \chi_{22}^{\prime\prime\prime}\left( x,\lambda\right)%
\end{array}
\right\vert ,
\end{equation*}
which are independent of $x$ and entire functions. This sort of calculation
gives $W_{1}\left( \lambda\right) =W_{2}\left( \lambda\right) .$ Now we may
introduce in consideration the characteristic function $W\left(
\lambda\right) $ as $W\left( \lambda\right) =W_{1}\left( \lambda\right) .$
\end{proof}

\begin{theorem}
\textit{The eigenvalues of the problem (1.1)-(1.9) are the zeros of the
function }$W\left( \lambda \right) .$
\end{theorem}

\begin{proof}
Let $W\left( \lambda \right) =0$. Then the functions $\phi _{11}\left(
x,\lambda \right) ,$ $\phi _{21}\left( x,\lambda \right) $ and $\chi
_{11}\left( x,\lambda \right) ,$ $\chi _{21}\left( x,\lambda \right) $ are
linearly dependent, i.e.,%
\begin{equation*}
k_{1}\phi _{11}\left( x,\lambda \right) +k_{2}\phi _{21}\left( x,\lambda
\right) +k_{3}\chi _{11}\left( x,\lambda \right) +k_{4}\chi _{21}\left(
x,\lambda \right) =0
\end{equation*}%
for some $k_{1}\neq 0$ or $k_{2}\neq 0$ or $k_{3}\neq 0$ or $k_{4}\neq 0$.
From this, it follows that $k_{3}\chi _{11}\left( x,\lambda \right)
+k_{4}\chi _{21}\left( x,\lambda \right) $ satisfies the boundary conditions
(1.2)-(1.3). Therefore%
\begin{equation*}
\left\{ 
\begin{array}{c}
k_{3}\chi _{11}\left( x,\lambda \right) +k_{4}\chi _{21}\left( x,\lambda
\right) ,\text{ \ }x\in \left[ -1,0\right) , \\ 
k_{3}\chi _{12}\left( x,\lambda \right) +k_{4}\chi _{22}\left( x,\lambda
\right) ,\text{ \ }x\in \left( 0,1\right]%
\end{array}%
\right.
\end{equation*}%
is an eigenfunction of the problem (1.1)-(1.9) corresponding to eigenvalue $%
\lambda .$

Now we let $u\left( x\right) $ be any eigenfunction corresponding to
eigenvalue $\lambda $, but $W\left( \lambda \right) \neq 0$. Then the
functions $\phi _{11},$ $\phi _{21},$ $\chi _{11},$ $\chi _{21}$ would be
linearly independent on $\left( 0,1\right] .$ Therefore $u\left( x\right) $
may be represented as%
\begin{equation*}
u\left( x\right) =\left\{ 
\begin{array}{c}
c_{1}\phi _{11}\left( x,\lambda \right) +c_{2}\phi _{21}\left( x,\lambda
\right) +c_{3}\chi _{11}\left( x,\lambda \right) +c_{4}\chi _{21}\left(
x,\lambda \right) ,\text{ \ }x\in \left[ -1,0\right) ; \\ 
c_{5}\phi _{12}\left( x,\lambda \right) +c_{6}\phi _{22}\left( x,\lambda
\right) +c_{7}\chi _{12}\left( x,\lambda \right) +c_{8}\chi _{22}\left(
x,\lambda \right) ,\text{ \ }x\in \left( 0,1\right] ,%
\end{array}%
\right.
\end{equation*}%
where at least one of the constants $c_{1},$ $c_{2},$ $c_{3},$ $c_{4},$ $%
c_{5},$ $c_{6},$ $c_{7}$ and $c_{8}$ is not zero. Considering the equations%
\begin{equation}
L_{\upsilon }\left( u\left( x\right) \right) =0\text{, \ }\upsilon =%
\overline{1,8}  \tag{3.6}
\end{equation}%
as a system of linear equations of the variables $c_{1},$ $c_{2},$ $c_{3},$ $%
c_{4},$ $c_{5},$ $c_{6},$ $c_{7},$ $c_{8}$ and taking (3.1)-(3.5) into
account, it follows that the determinant of this system is%
\begin{gather*}
\left\vert 
\begin{array}{cccccccc}
0 & 0 & L_{1}\chi _{11} & L_{1}\chi _{21} & 0 & 0 & 0 & 0 \\ 
0 & 0 & L_{2}\chi _{11} & L_{2}\chi _{21} & 0 & 0 & 0 & 0 \\ 
0 & 0 & 0 & 0 & L_{3}\phi _{12} & L_{3}\phi _{22} & 0 & 0 \\ 
0 & 0 & 0 & 0 & L_{4}\phi _{12} & L_{4}\phi _{22} & 0 & 0 \\ 
-\phi _{12}\left( 0\right) & -\phi _{22}\left( 0\right) & -\chi _{12}\left(
0\right) & -\chi _{22}\left( 0\right) & \phi _{12}\left( 0\right) & \phi
_{22}\left( 0\right) & \chi _{12}\left( 0\right) & \chi _{22}\left( 0\right)
\\ 
-\phi _{12}^{\prime }\left( 0\right) & -\phi _{22}^{\prime }\left( 0\right)
& -\chi _{12}^{\prime }\left( 0\right) & -\chi _{22}^{\prime }\left( 0\right)
& \phi _{12}^{\prime }\left( 0\right) & \phi _{22}^{\prime }\left( 0\right)
& \chi _{12}^{\prime }\left( 0\right) & \chi _{22}^{\prime }\left( 0\right)
\\ 
-\phi _{12}^{\prime \prime }\left( 0\right) & -\phi _{22}^{\prime \prime
}\left( 0\right) & -\chi _{12}^{\prime \prime }\left( 0\right) & -\chi
_{22}^{\prime \prime }\left( 0\right) & \phi _{12}^{\prime \prime }\left(
0\right) & \phi _{22}^{\prime \prime }\left( 0\right) & \chi _{12}^{\prime
\prime }\left( 0\right) & \chi _{22}^{\prime \prime }\left( 0\right) \\ 
-\phi _{12}^{\prime \prime \prime }\left( 0\right) & -\phi _{22}^{\prime
\prime \prime }\left( 0\right) & -\chi _{12}^{\prime \prime \prime }\left(
0\right) & -\chi _{22}^{\prime \prime \prime }\left( 0\right) & \phi
_{12}^{\prime \prime \prime }\left( 0\right) & \phi _{22}^{\prime \prime
\prime }\left( 0\right) & \chi _{12}^{\prime \prime \prime }\left( 0\right)
& \chi _{22}^{\prime \prime \prime }\left( 0\right)%
\end{array}%
\right\vert \\
=-W\left( \lambda \right) ^{3}\neq 0\text{.}
\end{gather*}%
Therefore, the system (3.6) has only the trivial solution $c_{i}=0$ $\left(
i=\overline{1,8}\right) $. Thus we get a contradiction, which completes the
proof.
\end{proof}

\section{\textbf{Asymptotic formulae for eigenvalues and fundamental
solutions}}

We start by proving some lemmas.

\begin{lemma}
\textit{Let }$\phi \left( x,\lambda \right) $\textit{\ be the solution of
Eq. (1.1) defined in Section 3, and let }$\lambda =s^{4}$, $s=\sigma +it$%
\textit{. Then the following integral equations hold for }$k=\overline{0,3}:$%
\begin{gather}
\frac{d^{k}}{dx^{k}}\phi _{11}\left( x,\lambda \right)  \notag \\
=\frac{a_{1}^{3}}{2s^{3}}\frac{d^{k}}{dx^{k}}\sin \frac{s\left( x+1\right) }{%
a_{1}}+\frac{a_{1}^{3}}{4s^{3}}\frac{d^{k}}{dx^{k}}e^{\frac{s\left(
x+1\right) }{a_{1}}}+\frac{a_{1}^{3}}{4s^{3}}\frac{d^{k}}{dx^{k}}e^{-\frac{%
s\left( x+1\right) }{a_{1}}}  \notag \\
+\frac{a_{1}^{3}}{2s^{3}}\int_{-1}^{x}\frac{d^{k}}{dx^{k}}\left( \sin \frac{%
s\left( x-y\right) }{a_{1}}-e^{\frac{s\left( x-y\right) }{a_{1}}}+e^{-\frac{%
s\left( x-y\right) }{a_{1}}}\right) q\left( y\right) \phi _{11}\left(
y,\lambda \right) dy.  \tag{4.1}
\end{gather}%
\begin{gather}
\frac{d^{k}}{dx^{k}}\phi _{12}\left( x,\lambda \right)  \notag \\
=\left( \frac{\phi _{12}\left( 0\right) }{2}-\frac{a_{2}^{2}\phi
_{12}^{\prime \prime }\left( 0\right) }{2s^{2}}\right) \frac{d^{k}}{dx^{k}}%
\cos \frac{sx}{a_{2}}+\left( \frac{a_{2}\phi _{12}^{\prime }\left( 0\right) 
}{2s}-\frac{a_{2}^{3}\phi _{12}^{\prime \prime \prime }\left( 0\right) }{%
2s^{3}}\right)  \notag \\
\times \frac{d^{k}}{dx^{k}}\sin \frac{sx}{a_{2}}+\left( \frac{\phi
_{12}\left( 0\right) }{4}+\frac{a_{2}\phi _{12}^{\prime }\left( 0\right) }{4s%
}+\frac{a_{2}^{2}\phi _{12}^{\prime \prime }\left( 0\right) }{4s^{2}}+\frac{%
a_{2}^{3}\phi _{12}^{\prime \prime \prime }\left( 0\right) }{4s^{3}}\right) 
\notag \\
\times \frac{d^{k}}{dx^{k}}e^{\frac{sx}{a_{2}}}+\left( \frac{\phi
_{12}\left( 0\right) }{4}-\frac{a_{2}\phi _{12}^{\prime }\left( 0\right) }{4s%
}+\frac{a_{2}^{2}\phi _{12}^{\prime \prime }\left( 0\right) }{4s^{2}}-\frac{%
a_{2}^{3}\phi _{12}^{\prime \prime \prime }\left( 0\right) }{4s^{3}}\right) 
\frac{d^{k}}{dx^{k}}e^{-\frac{sx}{a_{2}}}  \notag \\
+\frac{a_{2}^{3}}{2s^{3}}\int_{0}^{x}\frac{d^{k}}{dx^{k}}\left( \sin \frac{%
s\left( x-y\right) }{a_{2}}-e^{\frac{s\left( x-y\right) }{a_{2}}}+e^{-\frac{%
s\left( x-y\right) }{a_{2}}}\right) q\left( y\right) \phi _{12}\left(
y,\lambda \right) dy.  \tag{4.2}
\end{gather}%
\begin{gather}
\frac{d^{k}}{dx^{k}}\phi _{21}\left( x,\lambda \right)  \notag \\
=\frac{\beta _{1}a_{1}^{2}}{2s^{2}}\frac{d^{k}}{dx^{k}}\cos \frac{s\left(
x+1\right) }{a_{1}}+\frac{\beta _{2}a_{1}}{2s}\frac{d^{k}}{dx^{k}}\sin \frac{%
s\left( x+1\right) }{a_{1}}  \notag \\
+\left( \frac{\beta _{2}a_{1}}{4s}-\frac{\beta _{1}a_{1}^{2}}{4s^{2}}\right) 
\frac{d^{k}}{dx^{k}}e^{\frac{s\left( x+1\right) }{a_{1}}}-\left( \frac{\beta
_{2}a_{1}}{4s}+\frac{\beta _{1}a_{1}^{2}}{4s^{2}}\right) \frac{d^{k}}{dx^{k}}%
e^{-\frac{s\left( x+1\right) }{a_{1}}}  \notag \\
+\frac{a_{1}^{3}}{2s^{3}}\int_{-1}^{x}\frac{d^{k}}{dx^{k}}\left( \sin \frac{%
s\left( x-y\right) }{a_{1}}-e^{\frac{s\left( x-y\right) }{a_{1}}}+e^{-\frac{%
s\left( x-y\right) }{a_{1}}}\right) q\left( y\right) \phi _{21}\left(
y,\lambda \right) dy.  \tag{4.3}
\end{gather}%
\begin{gather}
\frac{d^{k}}{dx^{k}}\phi _{22}\left( x,\lambda \right)  \notag \\
=\left( \frac{\phi _{22}\left( 0\right) }{2}-\frac{a_{2}^{2}\phi
_{22}^{\prime \prime }\left( 0\right) }{2s^{2}}\right) \frac{d^{k}}{dx^{k}}%
\cos \frac{sx}{a_{2}}+\left( \frac{a_{2}\phi _{22}^{\prime }\left( 0\right) 
}{2s}-\frac{a_{2}^{3}\phi _{22}^{\prime \prime \prime }\left( 0\right) }{%
2s^{3}}\right)  \notag \\
\times \frac{d^{k}}{dx^{k}}\sin \frac{sx}{a_{2}}+\left( \frac{\phi
_{22}\left( 0\right) }{4}+\frac{a_{2}\phi _{22}^{\prime }\left( 0\right) }{4s%
}+\frac{a_{2}^{2}\phi _{22}^{\prime \prime }\left( 0\right) }{4s^{2}}+\frac{%
a_{2}^{3}\phi _{22}^{\prime \prime \prime }\left( 0\right) }{4s^{3}}\right) 
\notag \\
\times \frac{d^{k}}{dx^{k}}e^{\frac{sx}{a_{2}}}+\left( \frac{\phi
_{22}\left( 0\right) }{4}-\frac{a_{2}\phi _{22}^{\prime }\left( 0\right) }{4s%
}+\frac{a_{2}^{2}\phi _{22}^{\prime \prime }\left( 0\right) }{4s^{2}}-\frac{%
a_{2}^{3}\phi _{22}^{\prime \prime \prime }\left( 0\right) }{4s^{3}}\right) 
\frac{d^{k}}{dx^{k}}e^{-\frac{sx}{a_{2}}}  \notag \\
+\frac{a_{2}^{3}}{2s^{3}}\int_{0}^{x}\frac{d^{k}}{dx^{k}}\left( \sin \frac{%
s\left( x-y\right) }{a_{2}}-e^{\frac{s\left( x-y\right) }{a_{2}}}+e^{-\frac{%
s\left( x-y\right) }{a_{2}}}\right) q\left( y\right) \phi _{22}\left(
y,\lambda \right) dy\text{.}  \tag{4.4}
\end{gather}
\end{lemma}

\begin{proof}
Regard $\phi _{11}\left( x,\lambda \right) $ as the solution of the
following non-homogeneous Cauchy problem:%
\begin{equation*}
\left\{ 
\begin{array}{c}
-\left( a\left( x\right) u^{\prime \prime }\left( x\right) \right) ^{\prime
\prime }-s^{4}u\left( x\right) =q\left( x\right) \phi _{11}\left( x,\lambda
\right) , \\ 
\phi _{11}\left( -1,\lambda \right) =1,\text{ }\phi _{11}^{\prime }\left(
-1,\lambda \right) =0, \\ 
\phi _{11}^{\prime \prime }\left( -1,\lambda \right) =0,\text{ }\phi
_{11}^{\prime \prime \prime }\left( -1,\lambda \right) =0.%
\end{array}%
\right.
\end{equation*}%
Using the method of constant changing, $\phi _{11}\left( x,\lambda \right) $
satisfies%
\begin{gather}
\phi _{11}\left( x,\lambda \right) =\frac{a_{1}^{3}}{2s^{3}}\sin \frac{%
s\left( x+1\right) }{a_{1}}+\frac{a_{1}^{3}}{4s^{3}}e^{\frac{s\left(
x+1\right) }{a_{1}}}+\frac{a_{1}^{3}}{4s^{3}}e^{-\frac{s\left( x+1\right) }{%
a_{1}}}  \tag{4.1} \\
+\frac{a_{1}^{3}}{2s^{3}}\int_{-1}^{x}\left( \sin \frac{s\left( x-y\right) }{%
a_{1}}-e^{\frac{s\left( x-y\right) }{a_{1}}}+e^{-\frac{s\left( x-y\right) }{%
a_{1}}}\right) q\left( y\right) \phi _{11}\left( y,\lambda \right) dy. 
\notag
\end{gather}%
Then differentiating it with respect to $x$, we have $\left( 4.1\right) .$
The proof for $\left( 4.2\right) ,$ $\left( 4.3\right) $ and $\left(
4.4\right) $ is similar.
\end{proof}

\begin{lemma}
\textit{Let }$\lambda =s^{4},$ $s=\sigma +it$\textit{. Then the following
integral equations hold for }$k=\overline{0,3}:$%
\begin{equation}
\frac{d^{k}}{dx^{k}}\phi _{11}\left( x,\lambda \right) =O\left( \left\vert
s\right\vert ^{k-1}e^{\left\vert s\right\vert \frac{\left( x+1\right) }{a_{1}%
}}\right) \text{.}  \tag{4.5}
\end{equation}%
\begin{gather}
\frac{d^{k}}{dx^{k}}\phi _{12}\left( x,\lambda \right)  \notag \\
=\frac{a_{2}^{2}s^{2}\delta _{1}\phi _{11}^{\prime }\left( 0\right) }{2}%
\frac{d^{k}}{dx^{k}}\cos \frac{sx}{a_{2}}+\frac{a_{2}^{3}s\delta _{2}\phi
_{11}\left( 0\right) }{2}\frac{d^{k}}{dx^{k}}\sin \frac{sx}{a_{2}}  \notag \\
-\frac{a_{2}^{2}s^{2}\delta _{1}\phi _{11}^{\prime }\left( 0\right) }{4}%
\frac{d^{k}}{dx^{k}}\left( e^{\frac{sx}{a_{2}}}+e^{-\frac{sx}{a_{2}}}\right)
-\frac{a_{2}^{3}s\delta _{2}\phi _{11}\left( 0\right) }{4}\frac{d^{k}}{dx^{k}%
}\left( e^{\frac{sx}{a_{2}}}-e^{-\frac{sx}{a_{2}}}\right)  \tag{4.6} \\
+O\left( e^{\left\vert s\right\vert ^{k}\left( \frac{a_{1}x+a_{2}}{a_{1}a_{2}%
}\right) }\right) \text{.}  \notag
\end{gather}
\end{lemma}

\begin{gather*}
\frac{d^{k}}{dx^{k}}\phi _{21}\left( x,\lambda \right) \\
=\frac{\beta _{2}a_{1}}{2s}\frac{d^{k}}{dx^{k}}\sin \frac{s\left( x+1\right) 
}{a_{1}}+\frac{\beta _{2}a_{1}}{4s}\frac{d^{k}}{dx^{k}}\left( e^{\frac{%
s\left( x+1\right) }{a_{1}}}-e^{-\frac{s\left( x+1\right) }{a_{1}}}\right)
+O\left( \left\vert s\right\vert ^{k-2}e^{\left\vert s\right\vert \frac{x+1}{%
a_{1}}}\right) \text{.} \\
\frac{d^{k}}{dx^{k}}\phi _{22}\left( x,\lambda \right) \\
=\frac{a_{2}^{2}s^{2}\delta _{1}\phi _{21}^{\prime }\left( 0\right) }{2}%
\frac{d^{k}}{dx^{k}}\cos \frac{sx}{a_{2}}+\frac{a_{2}^{3}s\delta _{2}\phi
_{21}\left( 0\right) }{2}\frac{d^{k}}{dx^{k}}\sin \frac{sx}{a_{2}} \\
-\frac{a_{2}^{2}s^{2}\delta _{1}\phi _{21}^{\prime }\left( 0\right) }{4}%
\frac{d^{k}}{dx^{k}}\left( e^{\frac{sx}{a_{2}}}+e^{-\frac{sx}{a_{2}}}\right)
-\frac{a_{2}^{3}s\delta _{2}\phi _{21}\left( 0\right) }{4}\frac{d^{k}}{dx^{k}%
}\left( e^{\frac{sx}{a_{2}}}-e^{-\frac{sx}{a_{2}}}\right) \\
+O\left( e^{\left\vert s\right\vert ^{k-1}\left( \frac{a_{1}x+a_{2}}{%
a_{1}a_{2}}\right) }\right) .
\end{gather*}%
\textit{Each of these asymptotic formulae holds uniformly for }$x$\textit{\
as }$\left\vert \lambda \right\vert \rightarrow \infty .$

\begin{proof}
Let 
\begin{equation*}
F_{11}\left( x,\lambda \right) =e^{-\left\vert s\right\vert \frac{x+1}{a_{1}}%
}\phi _{11}\left( x,\lambda \right) \text{.}
\end{equation*}
It is easy to see that $F_{11}\left( x,\lambda \right) $ is bounded.
Therefore $\phi _{11}\left( x,\lambda \right) =O\left( e\right) .$
Substituting it into $(4.1)$ and differentiating it with respect to $x$ for $%
k=\overline{0,3},$ we obtain $(4.5).$ According to transmission conditions
(1.4)-(1.7) as $\left\vert \lambda \right\vert \rightarrow \infty ,$ we get%
\begin{gather*}
\phi _{12}\left( 0\right) =\phi _{11}\left( 0\right) ,\text{ }\phi
_{12}^{\prime }\left( 0\right) =\phi _{11}^{\prime }\left( 0\right) ,\text{ }%
\phi _{12}^{\prime \prime }\left( 0\right) =-s^{4}\delta _{1}\phi
_{11}^{\prime }\left( 0\right) ,\text{ } \\
\phi _{12}^{\prime \prime \prime }\left( 0\right) =-s^{4}\delta _{2}\phi
_{11}\left( 0\right) .
\end{gather*}%
Substituting these asymptotic formulae into $(4.2)$ for $k=0,$ we obtain%
\begin{gather}
\phi _{12}\left( x,\lambda \right) =\frac{a_{2}^{2}s^{2}\delta _{1}\phi
_{11}^{\prime }\left( 0\right) }{2}\cos \frac{sx}{a_{2}}+\frac{%
a_{2}^{3}s\delta _{2}\phi _{11}\left( 0\right) }{2}\sin \frac{sx}{a_{2}} 
\notag \\
-\frac{a_{2}^{2}s^{2}\delta _{1}\phi _{11}^{\prime }\left( 0\right) }{4}%
\left( e^{\frac{sx}{a_{2}}}+e^{-\frac{sx}{a_{2}}}\right) -\frac{%
a_{2}^{3}s\delta _{2}\phi _{11}\left( 0\right) }{4}\left( e^{\frac{sx}{a_{2}}%
}-e^{-\frac{sx}{a_{2}}}\right)  \notag \\
+\frac{a_{2}^{3}}{2s^{3}}\int_{0}^{x}\left( \sin \frac{s\left( x-y\right) }{%
a_{2}}-e^{\frac{s\left( x-y\right) }{a_{2}}}+e^{-\frac{s\left( x-y\right) }{%
a_{2}}}\right) q\left( y\right) \phi _{12}\left( y,\lambda \right) dy  \notag
\\
+O\left( e^{\left\vert s\right\vert \left( \frac{a_{1}x+a_{2}}{a_{1}a_{2}}%
\right) }\right) .  \tag{4.7}
\end{gather}%
Multiplying through by $\left\vert s\right\vert ^{-3}e^{-\left\vert
s\right\vert \left( \frac{a_{1}x+a_{2}}{a_{1}a_{2}}\right) },$ and denoting%
\begin{equation*}
F_{12}\left( x,\lambda \right) :=O\left( \left\vert s\right\vert
^{-3}e^{-\left\vert s\right\vert \left( \frac{a_{1}x+a_{2}}{a_{1}a_{2}}%
\right) }\right) \phi _{12}\left( x,\lambda \right) .
\end{equation*}%
Denoting $M:=\max_{x\in \left[ 0,1\right] }\left\vert F_{12}\left( x,\lambda
\right) \right\vert $ from the last formula, it follows that%
\begin{equation*}
M\left( \lambda \right) \leq \frac{M\left( \lambda \right) }{2\left\vert
s\right\vert ^{3}}\int_{0}^{x}q\left( y\right) dy+M_{0}
\end{equation*}%
for some $M_{0}>0$. From this, it follows that $M\left( \lambda \right)
=O\left( 1\right) $ as $\left\vert \lambda \right\vert \rightarrow \infty $,
so%
\begin{equation*}
\phi _{12}\left( x,\lambda \right) =O\left( \left\vert s\right\vert
^{3}e^{\left\vert s\right\vert \left( \frac{a_{1}x+a_{2}}{a_{1}a_{2}}\right)
}\right) .
\end{equation*}%
Substituting this back into the integral on the right side of $(4.7$) yields 
$\left( 4.6\right) $ for $k=0$. The other cases may be considered
analogically.
\end{proof}

Similarly one can establish the following lemma. for $\chi_{ij}\left(
x,\lambda\right) $ $\left( i=1,2,j=1,2\right) .$

\begin{lemma}
\textit{Let }$\lambda =s^{4}$\textit{, }$s=\sigma +it$\textit{. Then the
following integral equations hold for }$k=\overline{0,3}:$%
\begin{gather*}
\frac{d^{k}}{dx^{k}}\chi _{11}\left( x,\lambda \right) \\
=-\frac{a_{1}^{2}s^{2}\delta _{1}\chi _{12}^{\prime }\left( 0\right) }{2}%
\frac{d^{k}}{dx^{k}}\cos \frac{sx}{a_{1}}+\frac{a_{1}^{3}s\delta _{2}\chi
_{12}\left( 0\right) }{2}\frac{d^{k}}{dx^{k}}\sin \frac{sx}{a_{1}} \\
+\frac{a_{1}^{2}s^{2}\delta _{1}\chi _{12}^{\prime }\left( 0\right) }{4}%
\frac{d^{k}}{dx^{k}}\left( e^{\frac{sx}{a_{1}}}+e^{-\frac{sx}{a_{1}}}\right)
\\
+\frac{a_{1}^{3}s\delta _{2}\chi _{12}\left( 0\right) }{4}\frac{d^{k}}{dx^{k}%
}\left( e^{\frac{sx}{a_{1}}}-e^{-\frac{sx}{a_{1}}}\right) +O\left(
\left\vert s\right\vert ^{k+1}e^{\left\vert s\right\vert \left( \frac{%
a_{1}-a_{2}x}{a_{1}a_{2}}\right) }\right) \text{.} \\
\frac{d^{k}}{dx^{k}}\chi _{12}\left( x,\lambda \right) \\
=-\frac{a_{2}^{3}s}{2}\frac{d^{k}}{dx^{k}}\sin \frac{s\left( x-1\right) }{%
a_{2}}+\frac{a_{1}^{3}s\delta _{2}}{4}\frac{d^{k}}{dx^{k}}\left( e^{\frac{%
s\left( x-1\right) }{a_{2}}}-e^{-\frac{s\left( x-1\right) }{a_{2}}}\right)
+O\left( \left\vert s\right\vert ^{k+1}e^{\left\vert s\right\vert \frac{%
\left( 1-x\right) }{a_{2}}}\right) . \\
\frac{d^{k}}{dx^{k}}\chi _{21}\left( x,\lambda \right) \\
=-\frac{a_{1}^{2}s^{2}\delta _{1}\chi _{22}^{\prime }\left( 0\right) }{2}%
\frac{d^{k}}{dx^{k}}\cos \frac{sx}{a_{1}}+\frac{a_{1}^{3}s\delta _{2}\chi
_{22}\left( 0\right) }{2}\frac{d^{k}}{dx^{k}}\sin \frac{sx}{a_{1}} \\
+\frac{a_{1}^{2}s^{2}\delta _{1}\chi _{22}^{\prime }\left( 0\right) }{4}%
\frac{d^{k}}{dx^{k}}\left( e^{\frac{sx}{a_{1}}}+e^{-\frac{sx}{a_{1}}}\right)
+\frac{a_{1}^{3}s\delta _{2}\chi _{22}\left( 0\right) }{4}\frac{d^{k}}{dx^{k}%
}\left( e^{\frac{sx}{a_{1}}}-e^{-\frac{sx}{a_{1}}}\right) \\
+O\left( \left\vert s\right\vert ^{k+2}e^{\left\vert s\right\vert \left( 
\frac{a_{1}-a_{2}x}{a_{1}a_{2}}\right) }\right) \text{.} \\
\frac{d^{k}}{dx^{k}}\chi _{22}\left( x,\lambda \right) =-\frac{a_{2}^{2}s^{2}%
}{2}\frac{d^{k}}{dx^{k}}\cos \frac{s\left( x-1\right) }{a_{2}}+\frac{%
a_{2}^{2}s^{2}}{4}\frac{d^{k}}{dx^{k}}\left( e^{\frac{s\left( x-1\right) }{%
a_{1}}}-e^{-\frac{s\left( x-1\right) }{a_{1}}}\right) \\
+O\left( \left\vert s\right\vert ^{k+1}e^{\left\vert s\right\vert \frac{%
\left( 1-x\right) }{a_{2}}}\right) \text{,}
\end{gather*}%
\textit{where }$k=\overline{0,3}.$\textit{\ Each of these asymptotic
formulae holds uniformly for }$x.$
\end{lemma}

\begin{theorem}
\textit{Let }$\lambda =s^{4},$\textit{\ }$s=\sigma +it$\textit{. Then the
characteristic functions }$W_{i}\left( \lambda \right) $ $\left(
i=1,2\right) $\textit{\ have the following asymptotic formulae:}%
\begin{equation*}
W_{1}\left( \lambda \right) =W_{2}\left( \lambda \right) =O\left( \left\vert
s\right\vert ^{11}e^{2\left\vert s\right\vert \left( \frac{a_{1}+a_{2}}{%
a_{1}a_{2}}\right) }\right) \text{.}
\end{equation*}
\end{theorem}

\begin{proof}
Substituting the asymptotic equalities $\frac{d^{k}}{dx^{k}}\chi _{11}\left(
-1,\lambda \right) $ and $\frac{d^{k}}{dx^{k}}\chi _{21}\left( -1,\lambda
\right) $ into the representation of $W_{1}\left( \lambda \right) ,$ we get%
\begin{gather*}
W_{1}\left( \lambda \right) \\
=\left\vert 
\begin{array}{cccc}
0 & 0 & \chi _{11}\left( -1,\lambda \right) & \chi _{21}\left( -1,\lambda
\right) \\ 
0 & \beta _{2} & \chi _{11}^{\prime }\left( -1,\lambda \right) & \chi
_{21}^{\prime }\left( -1,\lambda \right) \\ 
0 & -\beta _{1} & \chi _{11}^{\prime \prime }\left( -1,\lambda \right) & 
\chi _{21}^{\prime \prime }\left( -1,\lambda \right) \\ 
-1 & 0 & \chi _{11}^{\prime \prime \prime }\left( -1,\lambda \right) & \chi
_{21}^{\prime \prime \prime }\left( -1,\lambda \right)%
\end{array}%
\right\vert \\
=\frac{a_{1}^{5}\delta _{1}\delta _{2}s^{3}}{8}\left( \chi _{12}^{\prime
}\left( 0\right) \chi _{22}\left( 0\right) -\chi _{12}\left( 0\right) \chi
_{22}^{\prime }\left( 0\right) \right) \\
\times \left( \left\vert 
\begin{array}{cccc}
0 & 0 & \cos \frac{s}{a_{1}} & e^{-\frac{s}{a_{1}}}-e^{\frac{s}{a_{1}}} \\ 
0 & \beta _{2} & -\frac{s}{a_{1}}\sin \frac{s}{a_{1}} & \frac{s}{a_{1}}%
\left( -e^{-\frac{s}{a_{1}}}-e^{\frac{s}{a_{1}}}\right) \\ 
0 & -\beta _{1} & -\frac{s^{2}}{a_{1}^{2}}\cos \frac{s}{a_{1}} & \frac{s^{2}%
}{a_{1}^{2}}\left( e^{\frac{s}{a_{1}}}-e^{\frac{s}{a_{1}}}\right) \\ 
-1 & 0 & -\frac{s^{3}}{a_{1}^{3}}\sin \frac{s}{a_{1}} & \frac{s^{3}}{%
a_{1}^{3}}\left( -e^{-\frac{s}{a_{1}}}-e^{\frac{s}{a_{1}}}\right)%
\end{array}%
\right\vert \right) \\
+\left\vert 
\begin{array}{cccc}
1 & 0 & \sin \frac{s}{a_{1}} & e^{-\frac{s}{a_{1}}}+e^{\frac{s}{a_{1}}} \\ 
0 & 0 & \frac{s}{a_{1}}\cos \frac{s}{a_{1}} & s\left( -e^{-\frac{s}{a_{1}}%
}+e^{\frac{s}{a_{1}}}\right) \\ 
0 & -1 & -\frac{s^{2}}{a_{1}^{2}}\sin \frac{s}{a_{1}} & s^{2}\left( e^{\frac{%
s}{a_{1}}}+e^{\frac{s}{a_{1}}}\right) \\ 
0 & 0 & -\frac{s^{3}}{a_{1}^{3}}\sin \frac{s}{a_{1}} & s^{3}\left( -e^{-%
\frac{s}{a_{1}}}+e^{\frac{s}{a_{1}}}\right)%
\end{array}%
\right\vert \\
+O\left( \left\vert s\right\vert ^{15}e^{2\left\vert s\right\vert \left( 
\frac{a_{1}+a_{2}}{a_{1}a_{2}}\right) }\right) \\
=0.
\end{gather*}%
Analogically, we can obtain the asymptotic formulae of $W_{2}\left( \lambda
\right) .$
\end{proof}

\begin{corollary}
\textit{The real eigenvalues of the problem (1.1)-(1.9) are bounded below.}
\end{corollary}

\begin{proof}
Putting $s^{2}=it^{2}$ $\left( t>0\right) $ in the above formulas, it
follows that 
\begin{equation*}
W\left( -t^{2}\right) \rightarrow \infty \text{ \textit{as} }t\rightarrow
\infty .
\end{equation*}
Therefore, $W\left( \lambda \right) \neq 0$ for $\lambda $ negative and
sufficiently large in modulus.
\end{proof}

\begin{corollary}
\textit{The non-real eigenvalues of the problem (1.1)-(1.9) are bounded
below and above.}
\end{corollary}

Now we can obtain the asymptotic approximation formulae for the eigenvalues
of the considered problem (1.1)-(1.9).

Since the eigenvalues coincide with the zeros of the entire function $%
W\left( \lambda\right) $, it follows that they have no finite limit.
Moreover, we know from Corollary 4.5 that all real eigenvalues are bounded
below. Hence, we may renumber them as $\lambda_{0}\leq\lambda_{1}\leq%
\lambda_{2}\leq...$, listed according to their multiplicity.

\begin{theorem}
\textit{The eigenvalues }$\lambda _{n}=s_{n}^{4}$\textit{, }$n=0,1,2,...$%
\textit{\ of the problem (1.1)-(1.9) have the following asymptotic formulae
for }$n\rightarrow \infty :$
\end{theorem}

\begin{equation*}
\sqrt[4]{\lambda _{n}^{\prime }}=\frac{a_{1}\pi \left( 2n-1\right) }{2}%
+O\left( \frac{1}{n}\right) \text{ \textit{and} }\sqrt[4]{\lambda
_{n}^{\prime \prime }}=\frac{a_{2}\pi \left( 2n+1\right) }{2}+O\left( \frac{1%
}{n}\right) .
\end{equation*}

\begin{proof}
By applying the well-known Rouch\'{e}'s theorem, which asserts that if $%
f\left( s\right) $ and $g\left( s\right) $ are analytic inside and on a
closed contour $C$, and $\left\vert g\left( s\right) \right\vert <\left\vert
f\left( s\right) \right\vert $ on $C$, then $f\left( s\right) $ and $f\left(
s\right) +g\left( s\right) $ have the same number zeros inside $C$ provided
that each zero is counted according to their multiplicity, we can obtain
these conclusions.
\end{proof}

%


\end{document}